%% file: main.tex
\documentclass[reqno, 12pt, psamsfonts]{amsart}
\usepackage{lineno}
\usepackage{amsthm,amsmath,amssymb}
\usepackage{setspace}
\usepackage[colorlinks=false]{hyperref}
\usepackage{fullpage}
\usepackage{datetime}
\usepackage{verbatim}
\usepackage{color}
\usepackage[utf8]{inputenc}
\usepackage{tikz}
\usetikzlibrary{calc,positioning,decorations.pathmorphing,decorations.pathreplacing}
\usepackage{tkz-graph}
\usepackage{enumerate}   

\newcommand{\calC}{\mathcal{C}}
\newcommand{\calA}{\mathcal{A}}
\newcommand{\ceil}[1]{\lceil{#1}\rceil}
\newcommand{\rarr}{\rightarrow}

\newcommand\modulo{\mathop{\textrm{\rm mod}}\nolimits}

\newtheorem{theorem}             {Theorem}[section] 
\newtheorem{lemma}     [theorem] {Lemma}        
\newtheorem{conjecture}[theorem] {Conjecture} 
\newtheorem{claim}     [theorem] {Claim}

\def\beq{\begin{equation}}\def\eeq{\end{equation}}
\def\beqn{\begin{eqnarray}}\def\eeqn{\end{eqnarray}}

\author[M. R. Cerioli]{Márcia R. Cerioli}
\author[C. G. Fernandes]{Cristina G. Fernandes}
\author[O. Lee]{Orlando Lee}
\author[C. N. Lintzmayer]{Carla N. Lintzmayer}
\author[G. O. Mota]{Guilherme O. Mota}
\author[C. N. da Silva]{Cândida N. da Silva}

\address{Instituto de Matemática, Universidade Federal do Rio de Janeiro, Rio de
Janeiro, Brazil}
\email{cerioli@cos.ufrj.br}

\address{Instituto de Matem\'atica e Estat\'{\i}stica, Universidade de S\~ao
Paulo, S\~ao Paulo, Brazil}
\email{cris@ime.usp.br}

\address{Instituto de Computação,  Universidade Estadual de Campinas, Campinas,
Brazil}
\email{\{lee,carlanl\}@ic.unicamp.br}

\address{Centro de Matem\'atica, Computa\c c\~ao e Cogni\c c\~ao, Universidade
Federal do ABC, Santo Andr\'e, Brazil}
\email{g.mota@ufabc.edu.br} 

\address{Departamento de Computação, Universidade Federal de São Carlos, São
Carlos, Brazil}
\email{candida@ufscar.br}

\thanks{%
  C. G. Fernandes was partially supported by CNPq (Proc.~308116/2016-0).
  O. Lee was supported by CNPq (Proc. 311373/2015-1 and 425340/2016-3) and FAPESP (Proc. 2015/11937-9)
  C. N. Lintzmayer was supported by FAPESP (Proc.~2016/14132-4).
  G. O. Mota was supported by FAPESP (Proc.~2013/03447-6) and CNPq
  (Proc.~459335/2014-6).
}

\title{Edge-magic labelings for constellations \\ and armies of caterpillars}

\begin{document}
\onehalfspacing
\date{\today, \currenttime}

\maketitle

\begin{abstract}
  Let $G=(V,E)$ be an $n$-vertex graph with $m$ edges.
  A function $f \colon V \cup E \rarr \{1,\ldots,n+m\}$ is an \emph{edge-magic
  labeling} of $G$ if $f$ is bijective and, for some integer $k$, we have
  $f(u)+f(v)+f(uv) = k$ for every edge $uv \in E$. 
  Furthermore, if $f(V) = \{1,\ldots,n\}$, then we say that $f$ is a \emph{super
  edge-magic labeling}. 
  A \emph{constellation}, which is a collection of stars, is \emph{symmetric}
  if the number of stars of each size is even except for at most one size.
  We prove that every symmetric constellation with an odd number of stars admits
  a super edge-magic labeling.
  We say that a caterpillar is of type $(r,s)$ if $r$ and $s$ are the sizes of
  its parts, where $r \leq s$.
  We also prove that every collection with an odd number of same-type
  caterpillars admits an edge-magic labeling.
\end{abstract}

\section{Introduction}

Let $G=(V,E)$ be an $n$-vertex graph with $m$ edges.
A function $f \colon V \cup E \rarr \{1,\ldots,n+m\}$ is an \emph{edge-magic
labeling} of $G$ if $f$ is bijective and, for some integer $k$, we have
$f(u)+f(v)+f(uv)= k$ for every edge $uv \in E$. 
The number $k$ is a \emph{magic constant} for $G$.  
Furthermore, if $f(V) = \{1,\ldots,n\}$, then $f$ is a \emph{super edge-magic
labeling}. 

A \emph{star} with $n$ vertices is a tree isomorphic to the complete bipartite
graph $K_{1,n-1}$.  
We denote by $|S|$ the size of a star $S$, which is its number of edges.
A \emph{caterpillar} is a tree composed by a central path and vertices directly
connected to this path.
Note that a star is also a caterpillar.

Edge-magic labelings were introduced by Kotzig and Rosa~\cite{1970-kotzig-rosa},
who proved that the following graphs are edge-magic: bipartite complete graphs
$K_{p,q}$ for all $p$, $q\geq 1$, cycles $C_n$ for all $n \geq 3$, paths $P_n$
for all $n \geq 2$, stars, caterpillars, and 1-regular graphs with an odd number
of edges.
Later~\cite{1972-kotzig-rosa} they proved that $K_n$ is edge-magic if and only
if $n = 2, 3, 5, 6$.
They also asked if any tree is edge-magic, a question that is open to this day.
The concept of super
edge-magic labelings was introduced by Enomoto, Lladó, Nakamigawa, and
Ringel~\cite{1998-enomoto-etal}, 
who proved that any $n$-vertex super edge-magic graph with $m$ edges must
satisfy $m \leq 2n-3$. 
Furthermore, they showed that $C_n$ is super edge-magic if and only if $n$ is
odd, $K_n$ is super edge-magic if and only if $n = 1,2,3$, and $K_{p,q}$ is
super edge-magic if and only if $p=1$ or $q=1$.
Note that this last result includes the stars.
They also verified that all trees with up to 15 vertices are super edge-magic,
and conjectured that any tree is super edge-magic.

In this paper we are interested in particular types of forests.
Kotzig~\cite{1971-kotzig} showed that if $G$ is a $3$-colorable edge-magic
graph, then any graph composed by the union of an odd number of copies of $G$ is
also edge-magic.
This directly implies that if $T$ is a path, a caterpillar, or a star, then a
forest composed by a union of an odd number of copies of $T$ is edge-magic.
Figueroa-Centeno, Ichishima and Muntaner-Batle~\cite{2002-figueroa-etal} showed
that a forest with $k$ copies of~$K_{1,n-1}$ is super edge-magic if $k$ is odd. 
They also proved that, if a forest consisting only of paths is super edge-magic,
then a graph composed by $k$ copies of such forest is also super edge-magic for
$k$ odd.
This implies that forests composed by the same stars (or paths) are super
edge-magic.
For more results about edge-magic labeling, we refer the reader to the
books~\cite{2017-book,2013-book}.

Our contribution is twofold: we investigate the problem of describing edge-magic
labelings in some forests of stars and some forests of caterpillars.
A \emph{constellation} is a forest whose components are stars.
We say a constellation is \emph{odd} if it consists of an odd number of stars.
Let $(S_{1},S_{2},\ldots,S_{k})$ denote a constellation with $k$ stars.
The following conjecture was posed in 2002.
\begin{conjecture}[Lee--Kong~\cite{2002-lee-kong}]\label{conj:LeeKong}
  Every odd constellation is super edge-magic.
\end{conjecture}
Lee and Kong~\cite{2002-lee-kong} proved that some constellations with less than
five stars are super edge-magic. 
In what follows, let $(S_{1},S_{2},\ldots,S_{k})$ be a constellation with
$|V(S_i)|=n_i$ for $1\leq i\leq k$.
Zhenbin and Chongjin~\cite{2013-zhenbin-chongjin} proved that, if
$k=2q+1$, $n_i = a$ for $1\leq i \leq q+1$, and $n_i = b$ for $q+2 \leq i \leq
2q+1$, then $(S_{1},S_{2},\ldots,S_{k})$ is super edge-magic.
Recently, Manickam, Marudai, and Kala~\cite{2016-manickam-etal} showed that, for
$r \geq 3$ odd, if $n_1,n_2,\ldots,n_k$ is an increasing sequence of positive
integers with $n_i = 1+(i-1)d$ for $1 \leq i \leq k$ and any $d$, then
$(S_{1},S_{2},\ldots,S_{k})$ is super edge-magic.
For instance, for $n_1=1,\dots,n_k=k$, the constellation $(S_1,S_2,\ldots,S_k)$
is super edge-magic.
We are interested in \emph{symmetric} constellations, which are those
such that the number of stars of each size is even except for at most one size.
In our first main result (Theorem~\ref{thm:constellations}), we prove that every
odd symmetric constellation has a super edge-magic labeling, providing a
positive result regarding Conjecture~\ref{conj:LeeKong}.

Our second result concerns forests of caterpillars.
Being a tree, a caterpillar is also a bipartite graph, so we say it is of
\emph{type $(r,s)$} if $r$ and $s$ are the sizes of the parts of its unique
bipartition, where $r \leq s$.
An \emph{army of caterpillars} is a forest whose components are caterpillars.
It is \emph{odd} if it has an odd number of components. 
An army of caterpillars is \emph{uniform} if all of its caterpillars are of the
same type.
In Theorem~\ref{thm:caterpillars}, we prove that every odd uniform army of
caterpillars has an edge-magic labeling.

This paper is organized as follows.
In Section~\ref{sec:constellations} we prove Theorem~\ref{thm:constellations},
which deals with constellations. The result concerning armies of caterpillars 
(Theorem~\ref{thm:caterpillars}) is proved in Section~\ref{sec:caterpillars}.

\section{Constellations}
\label{sec:constellations}

When dealing with super edge-magic labelings, the following result given by
Figueroa-Centeno, Ichishima, and Muntaner-Batle~\cite{2001-figueroa-etal} turns
out to be very useful and it will be used in this section.

\begin{lemma}[\cite{2001-figueroa-etal}]\label{lem:FigueroaEtAl}
  An $n$-vertex graph $G=(V,E)$ with $m$ edges is super edge-magic if and only
  if there exists a bijective function $f \colon V \rarr \{1,\ldots,n\}$ such
  that the set $L = \{f(u) + f(v) \colon uv \in E\}$ consists of $m$ consecutive
  integers.
  In such a case, $f$ extends to a super edge-magic labeling of~$G$. 
\end{lemma}

Let $\calC$ be an odd symmetric constellation, and let $S$ be the only star that
appears in~$\calC$ an odd number of times.
Let $p = |\calC|$ and $r$ be such that $p = 2r-1$. 
Let $S_1,\ldots,S_p$ be the stars in $\calC$ so that~$S_r = S$, $|S_i| =
|S_{p-i+1}|$ for every $i = 1,\ldots,r-1$, and $|S_i| \leq |S_{i+1}|$ for every
$i = 1,\ldots,r-2$.
Note that such an order exists because $\calC$ is odd and symmetric.
Let $c_i$ denote the central vertex of the star $S_i$.
From now on, we consider that any odd symmetric constellation is a sequence of
stars described as in such order.

A super edge-magic labeling $f$ of $\calC$ is \emph{standard} if $f(c_i) = i$
for $i=1,\ldots,p$ and, whenever $n > p$, the smallest sum in $L = \{f(u) + f(v)
\mid uv \in E(\calC)\}$ is $r+p+1$.

The next theorem shows that every odd symmetric constellation has a standard super
edge-magic labeling.
Figure~\ref{fig:example_constellation} shows an example of such labeling for a
constellation with seven stars.

  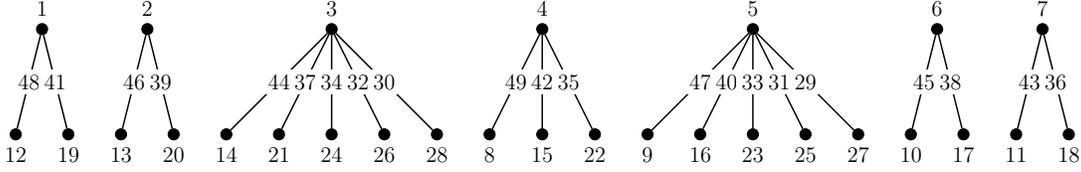
\begin{figure}
    \centering
    \scalebox{.7}{\input{example_constellation.tikz}}
    \caption{Example of a standard super edge-magic labeling with magic constant
    61 on an odd symmetric constellation.}
    \label{fig:example_constellation}
  \end{figure}

\begin{theorem}\label{thm:constellations}
  Every odd symmetric constellation has a standard super edge-magic labeling. 
\end{theorem}
\begin{proof}
  Let $\calC$ be an odd symmetric constellation.
  The proof is by induction on the number $n$ of vertices of $\calC$. 
  If $n = p$, then $\calC$ consists of only the centers of the stars. 
  Set $f(c_i) = i$ for $i=1,\ldots,p$.
  Trivially this is a standard super edge-magic labeling of $\calC$. 

  Suppose $n > p$, that is, $\calC$ contains at least one non-trivial star. 
  Let $p' > 0$ be the number of non-trivial stars in $\calC$.
  Note that $p'$ is odd if the central star $S$ is nontrivial; otherwise $p'$ is
  even.
  Let $\calC'$ be the constellation obtained from $\calC$ by removing a leaf
  from each non-trivial star in $\calC$.
  Call $u_1,\ldots,u_{p'}$ these removed leaves, in the order of their stars
  in~$\calC$, and let~$U = \{u_1,\ldots,u_{p'}\}$. 
  Note that $\calC'$ is odd and symmetric, and the order of the stars in $\calC$
  has the same properties in $\calC'$.
  Let $n'$ be the number of vertices in $\calC'$.
  Then $n' = n - p'$, and we can apply induction on $\calC'$ to obtain a
  standard super edge-magic labeling $f'$ for $\calC'$. 

  We define the standard labeling $f$ on $\calC$ as follows. 
  First, let $f(c_i) = f'(c_i) = i$ for $i = 1,\ldots,p$.  
  Second, for every vertex $v$ in $\calC'$ which is not a star center, let $f(v)
  = f'(v) + p'$.
  See Figure~\ref{fig:constellation_induction_step}.
  Now we describe the labels for vertices in $U$.
  Let $\ell = \ceil{(p'+1)/2}$.  
  The idea is to label the vertices in $U$ cyclically starting at $u_\ell$, 
  with labels $p+1,\ldots,p+p'$. 
  Set $f(u_j) = p+j-\ell+1$ for $j = \ell,\ldots,p'$
  and ${f(u_j) = (p+p'-\ell+1)+j = f(u_{p'}) + j}$ for $j = 1,\ldots,\ell-1$.
  Let us show that $f$ can be extended to a standard super edge-magic labeling. 
  That is, let us show that the set~${L = \{f(u) + f(v) \mid uv \in E(\calC)\}}$
  consists of consecutive integers starting from $r+p+1$.

  \begin{figure}
    \centering
    \scalebox{.7}{\input{constellation_induction_step.tikz}}
    \caption{Labeling for a constellation $\calC$ on the right obtained from the
    labeling of the constellation $\calC'$ on the left.}
    \label{fig:constellation_induction_step}
  \end{figure}
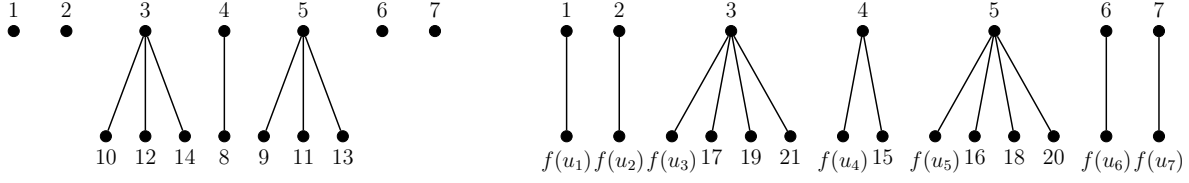

  As $f'$ is a standard super edge-magic labeling by induction, by
  Lemma~\ref{lem:FigueroaEtAl} we know that $L' = \{f'(u) + f'(v) \mid uv \in
  E(\calC')\}$ consists of consecutive integers starting from $r+p+1$. 
  By the way we defined $f$ on $V(\calC')$, the set $L_1 = \{f(u) + f(v) \mid uv
  \in E(\calC')\}$ consists of the consecutive integers starting from $r+p+1+p'$
  (the same consecutive integers plus $p'$). 
  Let us argue that $L_2 = \{f(u) + f(v) \mid uv \in E(\calC) \setminus
  E(\calC')\}$ consists of the $p'$ consecutive integers starting from $r+p+1$. 

  If $S_r$ is non-trivial, then $u_\ell$ is in $S_r$ and thus $f(c_r)+f(u_\ell)
  = r+p+1$. 
  For ${j = \ell+1,\ldots,p'}$, vertex $u_j$ is in $S_{r+j-\ell}$ (which is
  non-trivial by the order of the stars in~$\calC$), and therefore
  $f(c_{r+j-\ell})+f(u_j) = (r+j-\ell) + (p+j-\ell+1) = r+p+1 + 2(j-\ell)$. 
  Now, $u_1$ is in $S_q$ for $q = p-(r+p'-\ell) + 1 = r-p'+\ell$ and hence
  $f(c_q) + f(u_1) = q + (p+p'-\ell+2) = p+r+2$.
  For $j = 2,\ldots,\ell-1$, vertex $u_j$ is in $S_{q+j-1}$ (which is also
  non-trivial), and so $f(c_{q+j-1})+f(u_j) = (q+j-1) + (p+p'-\ell+j+1) =
  p+r+2j$.
  
  If $S_r$ is trivial, then $u_\ell$ is in $S_{r+1}$ and thus
  $f(c_{r+1})+f(u_\ell) = (r+1)+p+1 = p+r+2$. 
  For $j = \ell+1,\ldots,p'$, vertex $u_j$ is in $S_{r+1+j-\ell}$ (which is
  non-trivial by the order of the stars in~$\calC$), hence
  $f(c_{r+1+j-\ell})+f(u_j) = (r+1+j-\ell) + (p+j-\ell+1) = r+p+2 + 2(j-\ell)$. 
  Now, let~$q = p - (r+1+p'-\ell) + 1 = r-1-p'+\ell$. 
  For $j = 1,\ldots,\ell-1$, vertex $u_j$ is in the non-trivial star
  $S_{q+j-1}$, and so $f(c_{q+j-1})+f(u_j) = (q+j-1) + (p+p'-\ell+j+1) =
  p+r-1+2j$.

  This concludes the proof that $f$ can be extended to a standard super
  edge-magic labeling for $\calC$. 
\end{proof}

\section{Armies of caterpillars}
\label{sec:caterpillars}

Let $\calA$ be an odd uniform army of caterpillars of type $(r,s)$.
Let $p = |\calA|$ and let $C_1, \ldots, C_p$ be the caterpillars in $\calA$.
We will consider the following notation, which is depicted in
Figure~\ref{fig:army_of_caterpillars}.
Let $u_{i1},\ldots,u_{ir}$ denote the vertices of one part of the caterpillar
$C_i$ and $v_{i1},\ldots,v_{is}$ denote the vertices of the other part, so that
if $u_{ij}v_{ik} \in E(C_i)$, then ${u_{ij'}v_{ik'} \notin E(C_i)}$ for any $j'$
and $k'$ such that (i) $j'> j$ and $k' < k$ or (ii) $j' < j$ and $k' > k$.
Note that since caterpillars are planar graphs, there is always such an order 
of its vertices.
Furthermore, let $e_{i(r+s-1)}=u_{i1}v_{i1}$, let $e_{i(r+s-2)}$ be the next
edge in the considered order of vertices, and so on, until
$e_{i1}=u_{ir}v_{is}$.
Thus, the edges of $C_i$ are, in order, $e_{{i(r+s-1})}, e_{i(r+s-2)}, \dots,
e_{i1}$.  

\begin{figure}
  \centering
  \scalebox{.7}{\input{army_of_caterpillars.tikz}}
  \caption{An odd uniform army of caterpillars of type $(3,4)$.}
  \label{fig:army_of_caterpillars}
\end{figure}
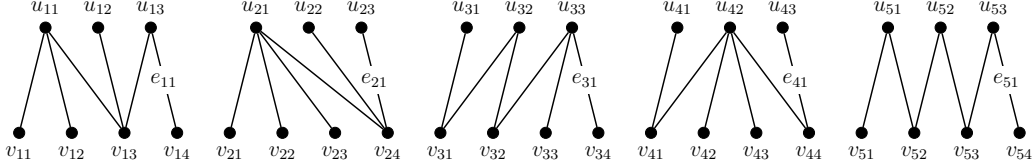

The next theorem shows that every odd uniform army of caterpillars has an
edge-magic labeling.
Figure~\ref{fig:example_caterpillars} shows an example of such a labeling for an
army of five caterpillars.

  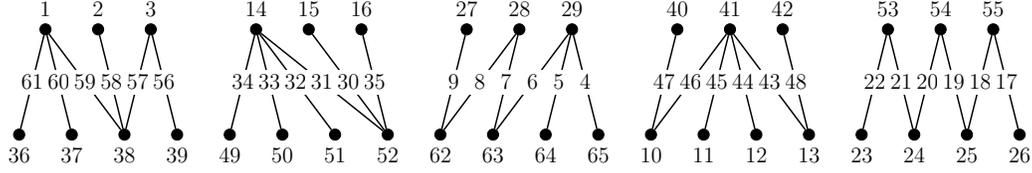
\begin{figure}
    \centering
    \scalebox{.7}{\input{example_caterpillars.tikz}}
    \caption{Example of an edge-magic labeling with magic constant 98 on an odd
    uniform army of five caterpillars of type $(3,4)$.}
    \label{fig:example_caterpillars}
  \end{figure}

\begin{theorem}\label{thm:caterpillars}
  Every odd uniform army of caterpillars has an edge-magic labeling.
\end{theorem}
\begin{proof}
  Let $\calA$ be an odd uniform army of caterpillars of type $(r,s)$ and let
  $C_1, \ldots, C_p$ be the caterpillars in $\calA$.
  For every ${1\leq i \leq p}$, denote the vertices of $C_i$ by
  $u_{i1},\ldots,u_{ir}$, $v_{i1},\ldots,v_{is}$ and the edges of $C_i$
  by~$e_{i1},\ldots,e_{i(r+s-1)}$ as described in the beginning of this section. 
    
  For simplicity, we use $x$ for the sum of the number of vertices and the
  number of edges in a caterpillar $C_i$, i.e., $x=2(r+s)-1$.
  We define an edge-magic labeling $f$ of $\calA$ as follows: 
  \begin{enumerate}[(i)]%
      \item $f(u_{ij}) = j + (i-1)x$ for any $1 \leq i \leq p$ and $1 \leq j
      \leq r$;\label{enum:i}
      \item $f(v_{ij}) = (2r+s-1) + j + \big((\frac{p-1}{2}+i-1) \modulo
      p\big)x$ for any $1\leq i \leq p$ and $1\leq j\leq s$;\label{enum:ii}
      \item $f(e_{ij}) = r + j + \big((2p-2i+1) \modulo p\big)x$ for any $1 \leq
      i \leq p$ and $1 \leq j \leq r+s-1$.\label{enum:iii}
  \end{enumerate}

  The intuition behind the above labeling can be better seen in the next
  diagram, which contains the labeling for the army of caterpillars depicted in
  Figure~\ref{fig:army_of_caterpillars}:
{\footnotesize\begin{equation*}
\begin{array}{lllllllll}
f(u_{11})=1  & f(u_{12})=2 & f(u_{13})=3  & f(e_{31})=4  & \ldots & f(e_{36})=9  & f(v_{41})=10 & \ldots & f(v_{44})=13\vspace{0.1cm} \\
f(u_{21})=14 & f(u_{22})=15 & f(u_{23})=16 & f(e_{51})=17 & \ldots & f(e_{56})=22 & f(v_{51})=23 & \ldots & f(v_{54})=26 \vspace{0.1cm}\\
f(u_{31})=27 & f(u_{32})=28 & f(u_{33})=29 & f(e_{21})=30 & \ldots & f(e_{26})=35 & f(v_{11})=36 & \ldots & f(v_{14})=39 \vspace{0.1cm}\\
f(u_{41})=40 & f(u_{42})=41 & f(u_{43})=42 & f(e_{41})=43 & \ldots & f(e_{46})=48 & f(v_{21})=49 & \ldots & f(v_{24})=52 \vspace{0.1cm}\\
f(u_{51})=53 & f(u_{52})=54 & f(u_{53})=55 & f(e_{11})=56 & \ldots & f(e_{16})=61 & f(v_{31})=62 & \ldots & f(v_{34})=65 \vspace{0.1cm}\\
\end{array}
\end{equation*}}%

  First we prove in Claim~\ref{claim:magic-cat} that, for every edge $uv$, we
  have ${f(u)+f(v)+f(uv) = k}$ for some constant $k=k(r,s,p)$ that depends only
  on $r$, $s$ and $p$.
  Then, in order to finish the proof, we show that all labels given by $f$ are
  different and lie between $1$ and $xp$. 

\begin{claim}\label{claim:magic-cat}
  For every $1 \leq i \leq p$ and every edge $uv$ in $C_i$,
  $$
    f(u)+f(v)+f(uv) = 4r + 2s + \left(\frac{3p-3}{2}\right)x.
  $$
\end{claim}
\noindent\emph{Proof of Claim~\ref{claim:magic-cat}}.
  We start by analyzing some particular labeling of a caterpillar.
  We say that a labeling of $C_i$ is \emph{well-behaved} if all the labels are
  different and it uses consecutive integers $a,a+1\dots,a+r$ respectively for
  the vertices $u_{i1},\ldots,u_{ir}$, consecutive integers $b,\dots,b+s$
  respectively for the vertices $v_{i1},\ldots,v_{is}$, and consecutive integers
  $c,\dots,c+(r+s-1)$ for the edges $e_{i1},\ldots,e_{i(r+s-1)}$.
  For any well-behaved labeling $f$ of a caterpillar $C_i$, a moment of thought
  shows that as $f(u)+f(v)+f(uv)$ is the same constant for every edge $uv$ of
  $C_i$.
  
  Clearly, the function $f$ defined above, restricted to $C_i$, which we denote
  by $f|_{C_i}$, is well-behaved. 
  Thus, as $f(u)+f(uv)+f(v)$ is the same constant for every edge $uv$ of $C_i$,
  it is enough to show that, for any $1 \leq i \leq p$,
  \begin{align*}
    f(u_{i1})&+f(v_{i1})+f(u_{i1}v_{i1}) \\
        &= f(u_{i1})+f(v_{i1})+f(e_{i(r+s-1)}) \\
        &= 1 + (i-1)x + (2r+s-1) + 1 + \left(\left(\frac{p-1}{2}+i-1\right)
        (\modulo p)\right)x \\
        &\quad + r + (r+s-1) + ((2p-2i+1) (\modulo p))x \\
        &= 4r + 2s + \left(\frac{3p-3}{2}\right)x.
  \end{align*}
  Hence the claim is proved.\qed

  We now proceed with the proof of the theorem.
  As discussed before, it remains to prove that all labels given by $f$ are
  different and lie between $1$ and $xp$. 
  
  Clearly, all $f(u_{ij})$, $f(v_{ij})$ and $f(e_{ij})$ are positive.
  It is also clear that $f(u_{ij}) \leq xp$, for $1 \leq i \leq p$ and $1 \leq j
  \leq r$. 
  Since $\big((\frac{p-1}{2}+i-1) \modulo p\big)x\leq (p-1)x$ and $(2r+s-1) +
  j\leq x$, in view of~\eqref{enum:ii} in the definition of~$f$, we have
  $f(v_{ij})\leq xp$, for $1 \leq i \leq p$ and $1 \leq j \leq s$. 
  Also, since $\big((2i-1) \modulo p\big)x\leq (p-1)x$ and $r + j \leq
  2r+s-1<x$, in view of~\eqref{enum:iii} in the definition of $f$, we have
  $f(e_{ij})\leq xp$, for $1 \leq i \leq p$ and $1 \leq j \leq r+s-1$. 
  
  Since the numbers multiplying $x$ in \eqref{enum:i}, \eqref{enum:ii}, and
  \eqref{enum:iii} are always between $0$ and $p-1$, we have that all
  $f(u_{ij})$ are different, and the same happens to all $f(v_{ij})$ and to all
  $f(e_{ij})$, in their respective ranges. 
  Thus, we only need to prove the following claim.
  
  \begin{claim}\label{claim:different-cat}
    The following holds for every $1\leq i$, $k$, $q\leq p$, every $1\leq j\leq
    r$, every $1\leq \ell\leq s$, and every~$1\leq t\leq r+s-1$:
    \begin{enumerate}[(a)]
      \item $f(v_{k\ell})\neq f(e_{qt}); $\label{dif:b}
      \item $f(e_{qt})\neq f(u_{ij})$; \label{dif:c}
      \item $f(u_{ij})\neq f(v_{k\ell})$. \label{dif:a}
    \end{enumerate} 
    \end{claim}
  
\noindent\emph{Proof of Claim~\ref{claim:different-cat}}.
  For simplicity, we let $\alpha=\left(\left(\frac{p-1}{2}+k-1\right) \modulo
  p\right)$ and $\beta=\big( (2p-2q+1) \modulo p\big)$.
  To see that \eqref{dif:b} holds, suppose for a contradiction that
  $$
  f(v_{k\ell})= (2r+s-1) + \ell + \alpha x = r + t + \beta x=f(e_{qt}).
  $$
  Then, recalling that $x=2r+2s-1$, we have
  $$
  (\alpha +1 - \beta)x  = r+ t + s - \ell.
  $$
  But since $1\leq \ell\leq s$ and $1\leq t\leq r+s-1$, we have $0<r+ t + s -
  \ell<x$, and thus the above equality does not hold.

  We proceed similarly to prove that \eqref{dif:c} holds.
  Suppose for a contradiction that 
  $$
  f(e_{qt}) = r + t + \beta x =  j + (i-1)x = f(u_{ij}).
  $$
  Then,
  $$
  \big((i-1) - \beta \big)x = r+t-j.
  $$
  But since $1\leq j\leq r$ and $1\leq t\leq r+s-1$, we have $0 < r + t - j <
  x$, and thus the above equality does not hold.

  Finally, to see that \eqref{dif:a} holds, suppose for a contradiction that
  $$
  f(u_{ij}) =  j + (i-1)x = (2r+s-1) + \ell + \alpha x = f(v_{k\ell}).
  $$
  Then, 
  $$
  \big((i-1) - \alpha \big)x = 2r + s -1 + \ell - j.
  $$
  But since $1\leq j\leq r$ and $1\leq \ell\leq s$, we have $0 < 2r + s - 1 +
  \ell - j < x$, and thus the above equality does not hold.\qed

  Therefore, since we proved Claim~\ref{claim:magic-cat}, the proof of
  Theorem~\ref{thm:caterpillars} is complete.
\end{proof}

\subsection*{Acknowledgement} This research was conducted while the authors were
attending the 1\textordmasculine~WoPOCA: ``Workshop Paulista em Otimização,
Combinatória e Algoritmos''. 
We would like to thank the organisers of this workshop for the productive
environment.

\bibliographystyle{plain}
\bibliography{bibfile}

\end{document}

%% file: example_constellation.tikz
%
%

\begin{tikzpicture}[scale=1]

    \GraphInit[vstyle=Classic]
    \SetGraphUnit{10}
    \tikzset{VertexStyle/.append style = {minimum size = 6pt, inner sep = 0pt}}
    \tikzset{LabelStyle/.append style = {fill=white, minimum size=0em,
    rectangle, inner sep=0.1cm}}

    \Vertex[L=1,x=0.5,y=2,Lpos=90]{A1}

    \Vertex[L=12,x=0,y=0,Lpos=-90]{A2}
    \Vertex[L=19,x=1,y=0,Lpos=-90]{A3}

    \Edge[label=48](A1)(A2)
    \Edge[label=41](A1)(A3)

\begin{scope}[xshift=2cm]
    \Vertex[L=2,x=0.5,y=2,Lpos=90]{A1}

    \Vertex[L=13,x=0,y=0,Lpos=-90]{A2}
    \Vertex[L=20,x=1,y=0,Lpos=-90]{A3}

    \Edge[label=46](A1)(A2)
    \Edge[label=39](A1)(A3)
\end{scope}

\begin{scope}[xshift=4cm]
    \Vertex[L=3,x=2,y=2,Lpos=90]{A1}

    \Vertex[L=14,x=0,y=0,Lpos=-90]{A2}
    \Vertex[L=21,x=1,y=0,Lpos=-90]{A3}
    \Vertex[L=24,x=2,y=0,Lpos=-90]{A4}
    \Vertex[L=26,x=3,y=0,Lpos=-90]{A5}
    \Vertex[L=28,x=4,y=0,Lpos=-90]{A6}

    \Edge[label=44](A1)(A2)
    \Edge[label=37](A1)(A3)
    \Edge[label=34](A1)(A4)
    \Edge[label=32](A1)(A5)
    \Edge[label=30](A1)(A6)
\end{scope}

\begin{scope}[xshift=9cm]
    \Vertex[L=4,x=1,y=2,Lpos=90]{A1}

    \Vertex[L=8,x=0,y=0,Lpos=-90]{A2}
    \Vertex[L=15,x=1,y=0,Lpos=-90]{A3}
    \Vertex[L=22,x=2,y=0,Lpos=-90]{A4}

    \Edge[label=49](A1)(A2)
    \Edge[label=42](A1)(A3)
    \Edge[label=35](A1)(A4)
\end{scope}

\begin{scope}[xshift=12cm]
    \Vertex[L=5,x=2,y=2,Lpos=90]{A1}

    \Vertex[L=9,x=0,y=0,Lpos=-90]{A2}
    \Vertex[L=16,x=1,y=0,Lpos=-90]{A3}
    \Vertex[L=23,x=2,y=0,Lpos=-90]{A4}
    \Vertex[L=25,x=3,y=0,Lpos=-90]{A5}
    \Vertex[L=27,x=4,y=0,Lpos=-90]{A6}

    \Edge[label=47](A1)(A2)
    \Edge[label=40](A1)(A3)
    \Edge[label=33](A1)(A4)
    \Edge[label=31](A1)(A5)
    \Edge[label=29](A1)(A6)
\end{scope}

\begin{scope}[xshift=17cm]
    \Vertex[L=6,x=0.5,y=2,Lpos=90]{A1}

    \Vertex[L=10,x=0,y=0,Lpos=-90]{A2}
    \Vertex[L=17,x=1,y=0,Lpos=-90]{A3}

    \Edge[label=45](A1)(A2)
    \Edge[label=38](A1)(A3)
\end{scope}

\begin{scope}[xshift=19cm]
    \Vertex[L=7,x=0.5,y=2,Lpos=90]{A1}

    \Vertex[L=11,x=0,y=0,Lpos=-90]{A2}
    \Vertex[L=18,x=1,y=0,Lpos=-90]{A3}

    \Edge[label=43](A1)(A2)
    \Edge[label=36](A1)(A3)
\end{scope}

\end{tikzpicture}

%% file: constellation_induction_step.tikz
%
%

\begin{tikzpicture}[scale=1]

    \GraphInit[vstyle=Classic]
    \SetGraphUnit{10}
    \tikzset{VertexStyle/.append style = {minimum size = 6pt, inner sep = 0pt}}
    \tikzset{LabelStyle/.append style = {fill=white, minimum size=0em,
    rectangle, inner sep=0.1cm}}

    \Vertex[L=1,x=0,y=2,Lpos=90]{A1}
    \Vertex[L=2,x=1,y=2,Lpos=90]{B1}
    \Vertex[L=3,x=2.5,y=2,Lpos=90]{C1}
    \Vertex[L=4,x=4,y=2,Lpos=90]{D1}
    \Vertex[L=5,x=5.5,y=2,Lpos=90]{E1}
    \Vertex[L=6,x=7,y=2,Lpos=90]{F1}
    \Vertex[L=7,x=8,y=2,Lpos=90]{G1}

    \Vertex[L=10,x=1.75,y=0,Lpos=-90]{C2}
    \Vertex[L=12,x=2.5,y=0,Lpos=-90]{C3}
    \Vertex[L=14,x=3.25,y=0,Lpos=-90]{C4}
    \Vertex[L=8,x=4,y=0,Lpos=-90]{D2}
    \Vertex[L=9,x=4.75,y=0,Lpos=-90]{E2}
    \Vertex[L=11,x=5.5,y=0,Lpos=-90]{E3}
    \Vertex[L=13,x=6.25,y=0,Lpos=-90]{E4}

    \Edge(C1)(C2)
    \Edge(C1)(C3)
    \Edge(C1)(C4)
    \Edge(D1)(D2)
    \Edge(E1)(E2)
    \Edge(E1)(E3)
    \Edge(E1)(E4)

\begin{scope}[xshift=10.5cm]
    \Vertex[L=1,x=0,y=2,Lpos=90]{A1}
    \Vertex[L=2,x=1,y=2,Lpos=90]{B1}
    \Vertex[L=3,x=3.125,y=2,Lpos=90]{C1}
    \Vertex[L=4,x=5.625,y=2,Lpos=90]{D1}
    \Vertex[L=5,x=8.125,y=2,Lpos=90]{E1}
    \Vertex[L=6,x=10.25,y=2,Lpos=90]{F1}
    \Vertex[L=7,x=11.25,y=2,Lpos=90]{G1}

    \Vertex[L=$f(u_1)$,x=0,y=0,Lpos=-90]{A2}
    \Vertex[L=$f(u_2)$,x=1,y=0,Lpos=-90]{B2}
    \Vertex[L=$f(u_3)$,x=2,y=0,Lpos=-90]{C2}
    \Vertex[L=17,x=2.75,y=0,Lpos=-90]{C3}
    \Vertex[L=19,x=3.5,y=0,Lpos=-90]{C4}
    \Vertex[L=21,x=4.25,y=0,Lpos=-90]{C5}
    \Vertex[L=$f(u_4)$,x=5.25,y=0,Lpos=-90]{D2}
    \Vertex[L=15,x=6,y=0,Lpos=-90]{D3}
    \Vertex[L=$f(u_5)$,x=7,y=0,Lpos=-90]{E2}
    \Vertex[L=16,x=7.75,y=0,Lpos=-90]{E3}
    \Vertex[L=18,x=8.5,y=0,Lpos=-90]{E4}
    \Vertex[L=20,x=9.25,y=0,Lpos=-90]{E5}
    \Vertex[L=$f(u_6)$,x=10.25,y=0,Lpos=-90]{F2}
    \Vertex[L=$f(u_7)$,x=11.25,y=0,Lpos=-90]{G2}

    \Edge(A1)(A2)
    \Edge(B1)(B2)
    \Edge(C1)(C2)
    \Edge(C1)(C3)
    \Edge(C1)(C4)
    \Edge(C1)(C5)
    \Edge(D1)(D2)
    \Edge(D1)(D3)
    \Edge(E1)(E2)
    \Edge(E1)(E3)
    \Edge(E1)(E4)
    \Edge(E1)(E5)
    \Edge(F1)(F2)
    \Edge(G1)(G2)
\end{scope}

\end{tikzpicture}

%% file: army_of_caterpillars.tikz
%
%

\begin{tikzpicture}[scale=1]

    \GraphInit[vstyle=Classic]
    \SetGraphUnit{10}
    \tikzset{VertexStyle/.append style = {minimum size = 6pt, inner sep = 0pt}}

    \Vertex[L=$u_{11}$,x=0.5,y=2,Lpos=90]{A1}
    \Vertex[L=$u_{12}$,x=1.5,y=2,Lpos=90]{A2}
    \Vertex[L=$u_{13}$,x=2.5,y=2,Lpos=90]{A3}

    \Vertex[L=$v_{11}$,x=0,y=0,Lpos=-90]{B1}
    \Vertex[L=$v_{12}$,x=1,y=0,Lpos=-90]{B2}
    \Vertex[L=$v_{13}$,x=2,y=0,Lpos=-90]{B3}
    \Vertex[L=$v_{14}$,x=3,y=0,Lpos=-90]{B4}

    \Edge(A1)(B1)
    \Edge(A1)(B2)
    \Edge(A1)(B3)
    \Edge(A2)(B3)
    \Edge(A3)(B3)
    \Edge[label=$e_{11}$](A3)(B4)

\begin{scope}[xshift=4cm]
    \Vertex[L=$u_{21}$,x=0.5,y=2,Lpos=90]{A1}
    \Vertex[L=$u_{22}$,x=1.5,y=2,Lpos=90]{A2}
    \Vertex[L=$u_{23}$,x=2.5,y=2,Lpos=90]{A3}

    \Vertex[L=$v_{21}$,x=0,y=0,Lpos=-90]{B1}
    \Vertex[L=$v_{22}$,x=1,y=0,Lpos=-90]{B2}
    \Vertex[L=$v_{23}$,x=2,y=0,Lpos=-90]{B3}
    \Vertex[L=$v_{24}$,x=3,y=0,Lpos=-90]{B4}

    \Edge[label=$e_{21}$](A3)(B4)
    \Edge(A1)(B1)
    \Edge(A1)(B2)
    \Edge(A1)(B3)
    \Edge(A1)(B4)
    \Edge(A2)(B4)

\end{scope}

\begin{scope}[xshift=8cm]
    \Vertex[L=$u_{31}$,x=0.5,y=2,Lpos=90]{A1}
    \Vertex[L=$u_{32}$,x=1.5,y=2,Lpos=90]{A2}
    \Vertex[L=$u_{33}$,x=2.5,y=2,Lpos=90]{A3}

    \Vertex[L=$v_{31}$,x=0,y=0,Lpos=-90]{B1}
    \Vertex[L=$v_{32}$,x=1,y=0,Lpos=-90]{B2}
    \Vertex[L=$v_{33}$,x=2,y=0,Lpos=-90]{B3}
    \Vertex[L=$v_{34}$,x=3,y=0,Lpos=-90]{B4}

    \Edge(A1)(B1)
    \Edge(A2)(B1)
    \Edge(A2)(B2)
    \Edge(A3)(B2)
    \Edge(A3)(B3)
    \Edge[label=$e_{31}$](A3)(B4)
\end{scope}

\begin{scope}[xshift=12cm]
    \Vertex[L=$u_{41}$,x=0.5,y=2,Lpos=90]{A1}
    \Vertex[L=$u_{42}$,x=1.5,y=2,Lpos=90]{A2}
    \Vertex[L=$u_{43}$,x=2.5,y=2,Lpos=90]{A3}

    \Vertex[L=$v_{41}$,x=0,y=0,Lpos=-90]{B1}
    \Vertex[L=$v_{42}$,x=1,y=0,Lpos=-90]{B2}
    \Vertex[L=$v_{43}$,x=2,y=0,Lpos=-90]{B3}
    \Vertex[L=$v_{44}$,x=3,y=0,Lpos=-90]{B4}

	\Edge[label=$e_{41}$](A3)(B4)
    \Edge(A1)(B1)
    \Edge(A2)(B1)
    \Edge(A2)(B2)
    \Edge(A2)(B3)
    \Edge(A2)(B4)
   
\end{scope}

\begin{scope}[xshift=16cm]
    \Vertex[L=$u_{51}$,x=0.5,y=2,Lpos=90]{A1}
    \Vertex[L=$u_{52}$,x=1.5,y=2,Lpos=90]{A2}
    \Vertex[L=$u_{53}$,x=2.5,y=2,Lpos=90]{A3}

    \Vertex[L=$v_{51}$,x=0,y=0,Lpos=-90]{B1}
    \Vertex[L=$v_{52}$,x=1,y=0,Lpos=-90]{B2}
    \Vertex[L=$v_{53}$,x=2,y=0,Lpos=-90]{B3}
    \Vertex[L=$v_{54}$,x=3,y=0,Lpos=-90]{B4}

    \Edge(A1)(B1)
    \Edge(A1)(B2)
    \Edge(A2)(B2)
    \Edge(A2)(B3)
    \Edge(A3)(B3)
    \Edge[label=$e_{51}$](A3)(B4)
\end{scope}

\end{tikzpicture}

%% file: example_caterpillars.tikz
%
%

\begin{tikzpicture}[scale=1]

    \GraphInit[vstyle=Classic]
    \SetGraphUnit{10}
    \tikzset{VertexStyle/.append style = {minimum size = 6pt, inner sep = 0pt}}
    \tikzset{LabelStyle/.append style = {fill=white, minimum size=0em,
    rectangle, inner sep=0.1cm}}

    \Vertex[L=1,x=0.5,y=2,Lpos=90]{A1}
    \Vertex[L=2,x=1.5,y=2,Lpos=90]{A2}
    \Vertex[L=3,x=2.5,y=2,Lpos=90]{A3}

    \Vertex[L=36,x=0,y=0,Lpos=-90]{B1}
    \Vertex[L=37,x=1,y=0,Lpos=-90]{B2}
    \Vertex[L=38,x=2,y=0,Lpos=-90]{B3}
    \Vertex[L=39,x=3,y=0,Lpos=-90]{B4}

    \Edge[label=61](A1)(B1)
    \Edge[label=60](A1)(B2)
    \Edge[label=59](A1)(B3)
    \Edge[label=58](A2)(B3)
    \Edge[label=57](A3)(B3)
    \Edge[label=56](A3)(B4)

\begin{scope}[xshift=4cm]
    \Vertex[L=14,x=0.5,y=2,Lpos=90]{A1}
    \Vertex[L=15,x=1.5,y=2,Lpos=90]{A2}
    \Vertex[L=16,x=2.5,y=2,Lpos=90]{A3}

    \Vertex[L=49,x=0,y=0,Lpos=-90]{B1}
    \Vertex[L=50,x=1,y=0,Lpos=-90]{B2}
    \Vertex[L=51,x=2,y=0,Lpos=-90]{B3}
    \Vertex[L=52,x=3,y=0,Lpos=-90]{B4}

    \Edge[label=35](A3)(B4)
    \Edge[label=34](A1)(B1)
    \Edge[label=33](A1)(B2)
    \Edge[label=32](A1)(B3)
    \Edge[label=31](A1)(B4)
    \Edge[label=30](A2)(B4)

\end{scope}

\begin{scope}[xshift=8cm]
    \Vertex[L=27,x=0.5,y=2,Lpos=90]{A1}
    \Vertex[L=28,x=1.5,y=2,Lpos=90]{A2}
    \Vertex[L=29,x=2.5,y=2,Lpos=90]{A3}

    \Vertex[L=62,x=0,y=0,Lpos=-90]{B1}
    \Vertex[L=63,x=1,y=0,Lpos=-90]{B2}
    \Vertex[L=64,x=2,y=0,Lpos=-90]{B3}
    \Vertex[L=65,x=3,y=0,Lpos=-90]{B4}

    \Edge[label=9](A1)(B1)
    \Edge[label=8](A2)(B1)
    \Edge[label=7](A2)(B2)
    \Edge[label=6](A3)(B2)
    \Edge[label=5](A3)(B3)
    \Edge[label=4](A3)(B4)
\end{scope}

\begin{scope}[xshift=12cm]
    \Vertex[L=40,x=0.5,y=2,Lpos=90]{A1}
    \Vertex[L=41,x=1.5,y=2,Lpos=90]{A2}
    \Vertex[L=42,x=2.5,y=2,Lpos=90]{A3}

    \Vertex[L=10,x=0,y=0,Lpos=-90]{B1}
    \Vertex[L=11,x=1,y=0,Lpos=-90]{B2}
    \Vertex[L=12,x=2,y=0,Lpos=-90]{B3}
    \Vertex[L=13,x=3,y=0,Lpos=-90]{B4}

	\Edge[label=48](A3)(B4)
    \Edge[label=47](A1)(B1)
    \Edge[label=46](A2)(B1)
    \Edge[label=45](A2)(B2)
    \Edge[label=44](A2)(B3)
    \Edge[label=43](A2)(B4)
   
\end{scope}

\begin{scope}[xshift=16cm]
    \Vertex[L=53,x=0.5,y=2,Lpos=90]{A1}
    \Vertex[L=54,x=1.5,y=2,Lpos=90]{A2}
    \Vertex[L=55,x=2.5,y=2,Lpos=90]{A3}

    \Vertex[L=23,x=0,y=0,Lpos=-90]{B1}
    \Vertex[L=24,x=1,y=0,Lpos=-90]{B2}
    \Vertex[L=25,x=2,y=0,Lpos=-90]{B3}
    \Vertex[L=26,x=3,y=0,Lpos=-90]{B4}

    \Edge[label=22](A1)(B1)
    \Edge[label=21](A1)(B2)
    \Edge[label=20](A2)(B2)
    \Edge[label=19](A2)(B3)
    \Edge[label=18](A3)(B3)
    \Edge[label=17](A3)(B4)
\end{scope}

\end{tikzpicture}